\documentclass[11pt]{amsart}
\usepackage{graphicx}
\usepackage{amssymb}
\newtheorem{thm}{Theorem}[section]

\newtheorem{lem}[thm]{Lemma}

\newtheorem{prop}[thm]{Proposition}
\theoremstyle{definition}
\newtheorem{defn}[thm]{Definition}
\theoremstyle{remark}

\numberwithin{equation}{section}
\newtheorem{exa}[thm]{Example}

\newcommand{\h}{\mathcal{H}}

\begin{document}

\title[Invariance of Fr\'echet frames under perturbation]
{Invariance of Fr\'echet frames under perturbation}%
\author{ Asghar Rahimi}
\address{Department of mathematics ,University of Maragheh, Maragheh, Iran}
 \email{asgharrahimi@yahoo.com}

\subjclass[2000]{ 42C15, 46A45} \keywords{Frame, Banach frame,
Fr\'echet frames, Fr\'echet spaces, Perturbation.}
\begin{abstract}

Motivating the perturbations of frames in Hilbert and Banach
spaces, in this paper we introduce the invariance of Fr\'echet
frames under perturbation. Also we show that for any Fr\'echet
spaces, there is a Fr\'echet frame and any element has a series
expansion.
\end{abstract}
\maketitle
\section{Introduction}
 Historically, theory of frames
appeared in the paper of Duffin and Schaeffer  in 1952. Around
1986, Daubechies, Grossmann, Meyer and others reconsidered ideas
of Duffin and Schaeffer \cite{DS}, and started to develop the
wavelet and frame theory. Frames for Banach spaces were introduced
by K. Gr\"{o}chenig \cite{gr} and subsequently many mathematicians
have contributed to this theory
\cite{AlST,CaChSt,CaHaLa,rahimibook}. There are some complete
spaces which are not Banach spaces ( like Fr\'echet spaces). This
was the main motivation for frames on Fr\'echet spaces. The
concept of Fr\'echet frames investigated by Pilipovi\'c and Stoeva
in \cite{phst1,phst2}. Like Hilbert and Banach spaces, we will
show that for any Fr\'echet space we can find a Fr\'echet frame
and every element in a Fr\'echet space has a series expansion
\cite{CaChSt}.

In this manuscript  we are interested in the problem of finding
conditions under which the perturbation of a Fr\'echet frame  is
also a Fr\'echet frame. The following result \cite{caschr98}, is
one of the most general and also typical results about frame
perturbations for the whole space $\h$ which generalizes the main
results in \cite{ole1,ole2}.
\begin{thm}
\cite{caschr98} Let $F :=\{f_i \}_{i\in I}$ be a frame for $\h$
with bounds $A$ and $B$, and $G:=\{g_i\}_{i\in I}$ is  a sequence
in $\h$. Suppose that there exist non-negative
$\lambda_1,\lambda_2$ and $\mu$  with $\lambda_2<1$ such that
\[
\|\sum_{i}c_i(f_i -g_i)\|\leq\lambda_1\|\sum_{i}c_i
f_i\|+\lambda_2\|\sum_{i}c_i g_i\|+\mu\|c\|
\]
for each finitly supported $c\in\ell^2( \mathbb{N}) $ and
$\lambda_1+ \frac{\mu}{\sqrt{A} }<1  $. Then $G$ is a frame for
$\h$ with bounds
$$
A(1-\frac{\lambda_1 +\lambda_2+\frac{\mu}{\sqrt{A}}
}{1+\lambda_2})^{2} \quad and\quad B(1+\frac{\lambda_1
+\lambda_2+\frac{\mu}{\sqrt{B}} }{1-\lambda_2})^{2}.
$$

\end{thm}

\section{Fr\'echet frames}
Throughout this paper, $(X, \|.\|)$ is a Banach space and $(X^*,
\|.\|^*)$ is its dual, $(\Theta, \||.\||)$ is a Banach sequence
space and $(\Theta^*, \||.\||^*)$ is the dual of $\Theta$.
\begin{defn}
A sequence space $X_d$ is called a Banach space of scalar valued
sequences or briefly a BK-space, if it is a Banach space and the
coordinate functionals are continuous on $X_d$ , i.e., the
relations $x_n = \{\alpha_j^{(n)}\}, x = \{\alpha_j\}\in X_d , lim
x_n=x$ imply $lim \alpha_j^{(n)}=\alpha_j$.
\end{defn}

\begin{defn}
 Let $X$ be a
Banach space and $X_d$ be a BK-space. A countable family
$\{g_i\}_{i\in I}$ in the dual $X^*$ is called an $X_d$ -frame for
$X$ if
\begin{enumerate}
\item $\{g_i(f )\}_{i=1}^{\infty}\in X_d , \forall f\in X$;
 \item  the norms $\|f\|_{X}$ and $\|\{g_i(f )\}_{i=1}^{\infty}\|_{X_d}$ are
 equivalent, i.e., there exist constants
$A,B > 0$ such that

\[
 A\|f\|_{X}\leq \|\{g_i(f )\}_{i=1}^{\infty}\|_{X_d}\leq B\|f\|_{X}, \quad\forall f\in X
 \]
 \end{enumerate}
 $A$ and $B$ are called $X_d$ -frame bounds. If at least (1) and the
upper condition in (2) are satisfied, $\{g_i\}_{i=1}^{\infty}$ is
called an $X_d$ -Bessel sequence for $X$.
\end{defn}
If $X$ is a Hilbert space and $X_d = \ell^2$, (2) means that
$\{g_i\}_{i=1}^{\infty}$ is a frame, and in this case it is well
known that there exists a sequence $\{f_i\}_{i=1}^{\infty}$ in $X$
such that $f =\sum_{i=1}^{\infty}\langle f,f_i\rangle g_i =
\sum\langle f,g_i\rangle f_i$.

Similar reconstruction formulas are not always available in the
Banach space setting. This is the reason behind the following
definition:
\begin{defn}\label{1.3}
 Let $X$ be a Banach space and $X_d$ a sequence space. Given a bounded
linear operator $S : X_d\to X$, and an $X_d$ -frame
$\{g_i\}\subset X^*$, we say that $(\{g_i\}_{i=1}^{\infty},S)$ is
a Banach frame for $X$ with respect to $X_d$ if
\begin{equation}\label{ban} S(\{g_i(f
)\}=f,\quad \forall f\in  X.\end{equation}
\end{defn}
 Note that (\ref{ban}) can be
considered as some kind of generalized reconstruction formula, in
the sense that it tells how to come back to $f\in X$ based on the
coefficients $\{g_i(f )\}_{i=1}^{\infty}$.

There is a relationship between these definitions, a Banach frame
is an atomic decomposition if and only if the unit vectors form a
basis for the space $X_d$. The following Proposition states this
result.
\begin{prop}\cite{CaHaLa} Let $X$ be a Banach space and
$X_d$ be a BK-space. Let $\{y_i\}_{i=1}^{\infty}\subseteq X^*$ and
$S:X_d\to X$ be given. Let $\{e_i\}_{i=1}^{\infty}$ be the unit
vectors in $X_d$. Then the following are equivalent:
\begin{enumerate}
\item $(\{y\}_i\}_{i=1}^{\infty}, S)$ is a Banach frame for $X$
with respect to $X_d$ and $\{e_i\}_{i=1}^{\infty}$ is a Schauder
basis for $X_d$. \item $(\{y_i\}_{i=1}^{\infty},
\{S(e_i)\}_{i=1}^{\infty})$ is an atomic decomposition for $X$
with respect to $X_d$.
\end{enumerate}
\end{prop}
It is known \cite{CaHaLa} that every separable Banach space has a
Banach frame.
\begin{thm}\label{2.5} Every separable Banach space has a Banach frame with
bounds $A=B=1$.
\end{thm}

The main motivation of Fr\'echet frames comes from some sequences
$\{g_i\}$ which are not Bessel sequences but they give rise to
series expansions. For Banach space $X$, let
$\{g_i\}_{i=1}^{\infty}\subseteq X^*$ be given and let there exist
$\{f_i\}_{i=1}^{\infty}\subseteq X$ such that the following series
expansion in $X$ holds
\begin{equation}\label{frech}
f =\sum_{i=1}^\infty g_i(f)f_i,\quad\forall f\in X.\end{equation}
Validity of (\ref{frech}) does not imply that
$\{g_i\}_{i=1}^{\infty}$ is a Banach frame for $X$ with respect to
the given sequence space. As one can see in the following
examples.

\begin{exa}
Let  $\{e_i\}_{i=1}^{\infty}$ be an orthonormal basis for the
Hilbert space $\h$. Consider the sequence
$\{g_i\}_{i=1}^{\infty}=\{e_1,e_1,2e_2,3e_3,4e_4,...\}$. This
sequence is not a Banach frame for $\h$ with respect to $\ell^2$.
However, the series expansion in $\h$ in the form (\ref{frech}) is
$\{f_i\}_{i=1}^{\infty}=\{e_1,0,\frac{1}{2}e_2,\frac{1}{3}e_3,
\frac{1}{4}e_4,... \}$.
\end{exa}
Validity of (\ref{frech}) implies that $\{g_i\}_{i=1}^{\infty}$ is
a Banach frame for $X$ with respect to the sequence space
$\big\{\{c_i\}_{i=1}^{\infty}:\sum _{i=1}^{\infty}c_i f_i\quad
converges\big\}.$

Recall, a complete locally convex space which has a countable
fundamental system of seminorms is called a Fr\'echet space.

 Let $\{Y_s, \|.\|_s\}_{s\in
\mathbb{N}}$ be a sequence of separable Banach spaces such that
\begin{equation}\label{ferc1}
\{0\}\neq\cap_{s\in \mathbb{N}}Y_s\subseteq ...\subseteq
Y_2\subseteq Y_1\subseteq Y_0
\end{equation}
\begin{equation}\label{ferc2}
\parallel.\parallel_{0}\leq\parallel.\parallel_{1}\leq\parallel.\parallel_{2}\leq
...
\end{equation}
\begin{equation}\label{ferc3}
Y_F:=\cap_{s\in \mathbb{N}}Y_s \quad is\quad dense\quad in\quad
Y_s \quad\forall s\in \mathbb{N}.
\end{equation}
Then $Y_F$ is a Fr\'echet space with the sequence of norms
$\|.\|_s$, $s\in\mathbb{N}$. We use the above sequences in two
cases: $Y_s=X_s$ and $Y_s=\Theta_s$. Let $\{X_s,\|.\|_s\}_{s\in
\mathbb{N}}$ and  $\{\Theta_s,\||.\||_s\}_{s\in \mathbb{N}}$ be
sequences of Banach and Banach sequence spaces, which satisfy
(\ref{ferc1})-(\ref{ferc3}). For fixed $s\in\mathbb{N}$, an
operator $V:\Theta_F\to X_F$ is $s$-bounded if there exist
constants $K_s>0$ such that $\|V\{c_i\}_{i=1}^{\infty}\|_s\leq
K_s\||\{c_i\}_{i=1}^{\infty}|\|_s$ for all
$\{c_i\}_{i=1}^{\infty}\in\Theta_F$. If $V$ is $s$-bounded for
every $s\in\mathbb{N}$, then $V$ is called $F$-bounded. Note that
an $F$-bounded operator is continuous but the converse dose not
hold in general. The book of R. Meise, D. Vogt is a very useful
text book about Fr\'echet spaces \cite{meise}.
\\
The Banach sequence space $\Theta$ is called solid if the
condition $\{c_i\}_{i=1}^{\infty}\in\Theta$ and $|d_i|\leq|c_i|$
for all $i\in\mathbb{N}$, imply that
$\{d_i\}_{i=1}^{\infty}\in\Theta$ and
$\||\{d_i\}_{i=1}^{\infty}\||_{\Theta}\leq
\||\{c_i\}\||_{\Theta}$.  A BK-space which contains all the
canonical vectors $e_i$ and for which there exists a constant
$\lambda\geq 1$ such that $$
\||\{c_i\}_{i=1}^{n}\||_{\Theta}\leq\lambda\||\{c_i\}_{i=1}^{\infty}\||_{\Theta},
\quad\forall n\in \mathbb{N},
\forall\{c_i\}_{i=1}^{\infty}\in\Theta
$$
will be called $\lambda$-BK-space. A BK-space is called a CB-space
if the set of the canonical vectors forms a basis.

\begin{defn}
Let $\{X_s,\|.\|_s\}_{s\in\mathbb{N}}$ be a family of Banach
spaces, satisfying (\ref{ferc1})-(\ref{ferc3}) and let
$\{\Theta_s,\||.|\|_s\}_{s\in\mathbb{N}}$ be a family of
$BK$-spaces, satisfying (\ref{ferc1})-(\ref{ferc3}). A sequence
$\{g_i\}_{i=1}^{\infty}\subseteq X^*_F$ is called a pre-Fr\'echet
frame ( a pre-F-frame) for $X_F$ with respect to $\Theta_F$ if for
every $s\in\mathbb{N}$ there exist constants $0<A_s\leq
B_s<\infty$ such that
\begin{equation}\label{ferc4}
\{g_i (f)\}_{i=1}^{\infty}\in\Theta_F,
\end{equation}
\begin{equation}\label{ferc5}
A_s\|f\|_s\leq\||\{g_i(f)\}_{i=1}^{\infty}|\|_s\leq B_s\|f\|_s,
\end{equation}
for all $f\in X_F$. The constants $A_s$ and $B_s$ are called lower
and upper bounds for $\{g_i\}_{i=1}^{\infty}$. The pre-F-frame is
called tight if $A_s=B_s$ for all $s\in\mathbb{N}$.
 Moreover, if there exists a $F$-bounded operator $S:\Theta_F\rightarrow X_F$ so that
$S(\{g_i(f)\}_{i=1}^{\infty})=f$ for all $f\in X_F$, then a
pre-F-frame $\{g_i\}$ is called a Fr\'echet frame ( or F-frame )
for $X_F$ with respect to $\Theta_F$ and $S$ is called an F-frame
operator of $\{g_i\}_{i=1}^{\infty}$. When (\ref{ferc4}) and at
least the upper inequality in (\ref{ferc5}) hold, then
$\{g_i\}_{i=1}^{\infty}$ is called a F-Bessel sequence for $X_F$
with respect to $\Theta_F$.
\end{defn}
Since $X_F$ is dense in $X_s$ for all $s\in\mathbb{N}$, $g_i$ has
a unique continuous extension on $X_s$, we show it by $g_i^s$.
Thus $g_i^s\in X_s^*$ and $g_{i}^s=g_i$ on $X_F$.

 The following Theorem gives some  conditions, under which an element can be expanded by some elementary vectors.
 \begin{thm}\cite{phst1,phst2}
Let $\{X_s,\|.\|_s\}_{s\in\mathbb{N}}$ be a family of Banach
spaces, satisfying (\ref{ferc1})-(\ref{ferc3}) and let
$\{\Theta_s,\||.|\|_s\}_{s\in\mathbb{N}}$ be a family of
CB-spaces, satisfying (\ref{ferc1})-(\ref{ferc3}) and we assume
that $\Theta_s^*$ is a CB-space for every $s\in\mathbb{N}$. Let
$\{g_i\}_{i=1}^{\infty}$ be a pre-F-frame for $X_F$ with respect
to $\Theta_F$. There exists a family
$\{f_i\}_{i=1}^{\infty}\subset X_F$ such that
\begin{enumerate}
\item $f=\sum_{i=1}^{\infty} g_i(f) f_i$ and
$g=\sum_{i=1}^{\infty} g(f_i)g_i$, $\forall f\in X_F $ and
$\forall g\in X_F^*$; \item $f=\sum_{i=1}^{\infty} g_i^s(f) f_i$
and $g=\sum_{i=1}^{\infty} g(f_i)g_i^s$, $\forall f\in X_s $ and
$\forall g\in X_s^*, \forall s\in\mathbb{N}$; \item for every
$s\in \mathbb{N}$, $\{f_i\}_{i=1}^{\infty}$ is a
$\Theta_s^*$-frame for $X_s^*$.
\end{enumerate}
if and only if there exists a continuous projection $U$ from
$\Theta_F$ onto its subspace $\big\{\{g_i(f)\}_{i=1}^{\infty}:f\in
X_F\big\}.$
\end{thm}
The following proposition shows that the pre-Fr\'echet Besselness
is equivalent to the F-boundedness of an operator.
\begin{prop}\label{TT}
Let $\{X_s,\|.\|_s\}_{s\in\mathbb{N}}$ be a family of Banach
spaces, satisfying (\ref{ferc1})-(\ref{ferc3}) and let
$\{\Theta_s,\||.|\|_s\}_{s\in\mathbb{N}}$ be a family of
CB-spaces, satisfying (\ref{ferc1})-(\ref{ferc3}) and we assume
that $\Theta_s^*$ is a CB-space for every $s\in\mathbb{N}$. The
family $\{g_i\}_{i=1}^{\infty}\subseteq X^*_F$ is a pre-Fr\'echet
Bessel sequences for $X_F$ with respect to $\Theta_F$ if and only
if the operator $T:\Theta^*_F\rightarrow X^*_F$ defined by $
T\{d_i\}_{i=1}^{\infty}=\sum_{i=1}^{\infty} d_i g_i$ is well
defined and $\| T\|_s\leq B_s $, for all $s\in \mathbb{N}$.
\end{prop}
\begin{proof}
First, suppose that $\{g_i\}_{i=1}^{\infty}\subseteq X^*_F$ is a
pre-Fr\'echet Bessel sequences for $X_F$ with respect to
$\Theta_F$. Define the operator $$R:X_F\rightarrow\Theta_F$$ by
$$Rf=\{g_i(f)\}_{i=1}^{\infty}.$$ Since $\{g_i(f)\}_{i=1}^{\infty}$
is a pre-Fr\'echet Bessel sequence , so  $\|R\|_s\leq B_s$ for
every $s\in\mathbb{N}$. The adjoint of $R$ is in the form
$R^*:\Theta^*_F\rightarrow X^*_F$ and
$$R^*(e_j)f=e_j(Rf)=e_j(\{g_i f\}_{i=1}^{\infty})=g_j f$$, so $R^*
e_j=g_j$. Put $T=R^*$, then $\|T\|_s\leq B_s$ and $$
T\{d_i\}_{i=1}^{\infty}=T(\sum_{i=1}^{\infty} d_i
e_i)=\sum_{i=1}^{\infty} d_i Te_i=\sum_{i=1}^{\infty} d_i R^*
e_i=\sum_{i=1}^{\infty} d_i g_i.$$
 Conversely, suppose the operator $T:\Theta^*_F\rightarrow X^*_F$ defined by $
T\{d_i\}_{i=1}^{\infty}=\sum_{i=1}^{\infty} d_i g_i$ is well
defined and $s$-bounded for all $ s\in\mathbb{N} $ . It is clear
that $Te_i=g_i$ and $T^*:X_F^{**}\rightarrow \Theta^{**}_F$ is
$(T^*f)e_i=f(Te_i)=f(g_i)$. Therefore $\{g_i
(f)\}_{i=1}^{\infty}=\{(T^*f)e_i\}_{i=1}^{\infty}$ , i.e. $\{g_i
(f)\}_{i=1}^{\infty}\in\Theta_F$. The $s$- boundedness of $T$
imply that $\||\{g_i (f)\}_{i=1}^{\infty}|\|_s\leq B_s$.
\end{proof}

Similar to Theorem \ref{2.5}, the following Theorem shows that for
any Fr\'echet space $X_F$ we can construct a Fr\'echet frame.
\begin{thm}
Let $\{X_s,\|.\|_s\}_{s\in\mathbb{N}}$ be a family of separable
Banach spaces satisfying (\ref{ferc1})-(\ref{ferc3}). Let $X_F:=
\bigcap_{s\in\mathbb{N}} X_s$. Then $X_F$ can be equipped with a
Fr\'echet frame with respect to an appropriately sequence space.
\end{thm}
\begin{proof}
Since $X_s$ is a separable Banach space for any $s\in\mathbb{N}$,
there is a sequence $\{x_i^s\}_{i=1}^{\infty} \subseteq X_s$, such
that $\overline{\{x_i^s\}_{i=1}^{\infty}}= X_s$. For any $f_s\in
X_s$ there is a subsequence $x^s_{k_i}\rightarrow f_s$ as
$i\rightarrow\infty$. By Hahn-Banach Theorem, there is $g_i^s\in
X^*_s$ such that $g_i^s(x_i^s)=\|x_i^s\|$ and $\|g_i^s\|=1$. Now,
$$\|x_{k_i}^s\|=|g_{k_i}^s(x_{k_i})|\leq\|g_{k_i}^s(f^s)\|+\|f^s-x_{k_i}^s\|$$,
so $\|f^s\|\leq\sup\|g_i(f^s)\|.$ Since we also have
$\|f^s\|\geq\sup_i\|g_i(f^s)\|,$ therefore
$$\|f^s\|=\sup_{i}\|g_i^s(f^s)\|,\quad\forall f^s\in X_s.
$$
Let $\Theta_s\subseteq\ell^\infty$ and
$\Theta_s=\big\{\{g_i(f^s)\}: f^s\in X_s\big\}$, then
$\{\Theta_s\}_{s\in \mathbb{N}  }\Theta_s$ satisfies
(\ref{ferc1})-(\ref{ferc3}). Let $\Theta_F:=\bigcap_{s\in
\mathbb{N}}$ and $S(\{g_i(f)\})=f$. Then $(\{g_i^s\}, S)$ is a
Fr\'echet frame for $X_F$ with respect $\Theta_F$.
\end{proof}

\section{Perturbation of Fr\'echet Frames}
Like the results about $p$-frames \cite{stoeva09}, Banach frames
and atomic decompositions \cite{ole,jkv}, generalized frames
\cite{NaRa1}, continuous frames \cite{ranade}, we study
perturbations of Fr\'echet frames. We need the following
assertion.
\begin{lem}\label{1.1}
\cite{kato} Let $U:X\rightarrow X$ be a linear operator and assume
that there exist constants $\lambda_1,\lambda_2\in[0,1[$ such
that$$ \|x-Ux\|\leq\lambda_1\|x\|+\lambda_2\|Ux\|,\quad\forall
x\in X.
$$ Then $U$ is invertible and $$
\frac{1-\lambda_1}{1+\lambda_2}\|x\|\leq\|Ux\|\leq\frac{1+\lambda_1}{1-\lambda_2}\|x\|
$$$$
\frac{1-\lambda_2}{1+\lambda_1}\|x\|\leq\|U^{-1}x\|\leq\frac{1+\lambda_2}{1-\lambda_1}\|x\|
$$
for all $x\in X$.
\end{lem}

The simplest  assertion about perturbation of F-Bessel sequences
is:

\begin{prop}
Let $X_F$ be a Fr\'echet space satisfying (1)-(3) and let
$\Theta_F$ be a Fr\'echet sequence space satisfying (1)-(3) so
that $\Theta_s$ is a reflexive $CB$-space for every $s$. Let
$\{g_i\}$ be an Fr\'echet Bessel sequence for $X_F$ with respect
to $\Theta_F$ with bounds $B_s$ and let $\{f_i\}\subset X_F^*$.

Assume that $\{(g_i - f_i )(f)\} \in \Theta_F$, $\forall f\in
X_F$, and $\exists \, \widetilde{\mu}_s \geq 0$ such that
\begin{equation}\label{b2b}
\|| \{(g_i - f_i )(f)\}| \|_{s} \le \widetilde{\mu}_s \| f\|_{s},
\ \forall f\in X_F
\end{equation}
(i.e.  $\{g_i-f_i\}$ is an Fr\'echet Bessel sequence for $X_F$
w.r.t. $\Theta_F$). Then $\{f_i\}$ is an $F$-Bessel sequence for
$X_F$ w.r.t. $\Theta_F$ with bounds $B_s+\widetilde{\mu}_s$.

The converse also holds.
\end{prop}
\begin{proof}
It is clear that $\{(f_i )(f)\}= \{(-g_i+ f_i )(f)\}+ \{(g_i - f_i
)(f)\}\in\Theta_F $ for all $f\in X_F$, also
$$
|\|\{f_i(f)\}|\|_s\leq|\|\{(g_i - f_i )(f)\}|\|_s+|\|\{(g_i - f_i
)(f)\}|\|_s\leq (B_s+\widetilde{\mu}_s)\|f\|
$$
for all $s\in\mathbb{N}, f\in X_F $.
\end{proof}

The following Theorem generalizes a Theorem of \cite{jkv} to
Fr\'echet frames and gives a necessary and sufficient condition
for the stability of Fr\'echet frames.
\begin{thm}
Let $\big(\{g_i\}, S\big)$ be a Fr\'echet frame for $X_F$ with
respect to $\Theta_F$ with bounds $A_s$ and $B_s$. Let
$\{h_i\}\subseteq X^*_F$ such that $\{h_i(f)\}\in\Theta_F$ for all
$f\in X_F$ and let $D:\Theta_F\rightarrow\Theta_F$ be a continuous
linear operator such that $D\{h_n(f)\}=\{g_n(f)\}$, $f\in X_F$.
Then there exists an operator $V:\Theta_F\rightarrow X_F$ such
that $(\{h_n\}, V)$ is a Fr\'echet frame if and only if for each
$s\in\mathbb{N}$ there exists $\lambda_{s}>0$ such that
\begin{equation}\label{612}
\||\{(g_n -h_n)(f)\}|\|_{s}\leq \lambda_s
min\{\||\{g_n(f)\}|\|_{s},\|| \{(h_n(f))\}|\|_{s}\}
\end{equation} for all $f\in X_F.$
\end{thm}
\begin{proof}
Suppose there exists $\lambda_s$ for every $s\in\mathbb{N}$ such
that (\ref{612}) holds. By assumption, $\{h_i(f)\}\in\Theta_F$ for
all $f\in X_F$. For any $f\in X_F$ and $s\in\mathbb{N}$,
\begin{eqnarray*}
A_s\|f\|_{s}&\leq&\||\{g_n(f)\}|\|_s\\&\leq&
\||(g_n-h_n)(f)\}|\|_s+\||\{h_n(f)\} |\|_s\\&\leq&
\lambda_s\||\{h_n(f)\}|\|_s+\||\{h_n(f)\}|\|_s
\\&=&(1+\lambda_s)\||\{h_n(f)\}|\|_s\\&\leq&
(1+\lambda_s)\big(\||\{(g_n-h_n)(f)\}|\|_s+\||\{g_n(f)\}|\|_s\big)\\&\leq&
(1+\lambda_s)^2\||\{g_n(f)\}|\|_s\\&\leq& (1+\lambda_s)^2
B_s\|f\|_s.
\end{eqnarray*}
So we have
$$
\frac{A_s}{1+\lambda_s}\|f\|_s\leq\||\{h_n(f)\}|\|_s\leq(1+\lambda_s)B_s\|f\|_s.
$$
Let $V=SD$. Then $V(\{h_n(f)\})=SD(\{h_n(f)\})=f$, i.e.
$\big(\{h_n\}, V\big)$ is a Fr\'echet frame for $X_F$ with respect
to $\Theta_F$.\\
Conversely, suppose $\big(\{g_i\}, S\big)$ and $\big(\{h_n\},
V\big)$ are  Fr\'echet frames for $X_F$ with respect to $\Theta_F$
 with bounds $A_s, B_s$ and $A'_s,B'_s$ , respectively. Then, by using the inequalities, we get
$$
\||\{(g_n-h_n)(f)\}|\|_{s}\leq\big(1+\frac{B'_s}{A_s}\big)\||\{g_n(f)\}|\|_{s},
\quad f\in X_F
$$
and
$$
\||\{(g_n-h_n)(f)\}|\|_{s}\leq\big(1+\frac{B_s}{A'_s}\big)\||\{h_n(f)\}|\|_{s},
\quad f\in X_F.
$$
Choose $\lambda_s:=Max\{1+\frac{B'_s}{A_s},1+\frac{B_s}{A'_s}\}$,
therefore (\ref{612}) holds for any $f\in X_F$.
\end{proof}

\begin{thm}
Let $\big(\{g_i\}, S\big)$ be a Fr\'echet frames for $X_s$ with
respect to $\Theta_s$. Let $\{h_i\}\in X^*_F$ such that
$\{h_i(f)\}\in\Theta_F$ for all $f\in X_F$ and let
$D:\Theta_F\rightarrow\Theta_F$ be a continuous ( or F-bounded)
linear operator such that $D\{g_n(f)\}=\{h_n(f)\}$, $f\in X_F$.
Let $\{\alpha_n\}$ and $\{\beta_n\}$ be sequences of positive
numbers for which $0<\inf \alpha_n\leq \sup\alpha_n<\infty$ and
$0<\inf \beta_n\leq \sup\beta_n<\infty$. If for any
$s\in\mathbb{N}$ there exist non negative scalers
$\lambda_s,\mu_s\in[0,1[$ and $\gamma_s$ such that
\begin{enumerate}
 \item  $\|S\|_s\gamma_s<(1-\lambda_s) (\inf\alpha_n)$
 \item  $\||\{(\alpha_n g_n-\beta_n h_n)f\}|\|_s\leq\lambda_s\||\{(\alpha_n
 g_n)f\}|\|_s+\mu_s\||\{(\beta_n h_n)f\}|\|_s+\gamma_s\|f\|_s, \forall
 f\in X_F,$
 \end{enumerate}
then there is a F-bounded operator $V:\Theta_F\rightarrow X_F$
such that $(\{h_i\}, V)$ is a Fr\'echet frame for $X_F$ with
respect to $\Theta_F$.
\end{thm}

\begin{proof}
Let $W:X_F\rightarrow\Theta_F$ with $Wf:=\{g_i(f)\}$ for all $f\in
X_F$. Then $SW:\Theta_F\rightarrow\Theta_F$ is an identity
operator and for all $s\in\mathbb{N}$
\[ \|f\|_s=\|SWf\|_s\leq\|S\|_s\||\{g_i(f)\}|\|_s.
\]
Now,
\begin{eqnarray*}
\||\{(\beta_n h_n)f\}|\|_s&\leq&\||\{(\alpha_n
 g_n)f\}|\|_s+\||\{(\alpha_n g_n-\beta_n h_n)f\}|\|_s\\
&\leq&\||\{(\alpha_n
 g_n)f\}|\|_s+\lambda_s\||\{(\alpha_n
 g_n)f\}|\|_s\\
 &+&\mu_s\||\{(\beta_n h_n)f\}|\|_s+\gamma_s\|f\|_s
\end{eqnarray*}
for all $s\in\mathbb{N}$ and $f\in X_F$. Therefore
\[
(1-\mu_s)\||\{(\beta_n
h_n)f\}|\|_s\leq\big((1+\lambda_s)\|W\|_s(sup\alpha_n)+\gamma_s\big)\|f\|_s,
\]
also
\begin{eqnarray*}
(1-\mu_s)\big(\inf \beta_n\big)\||\{(
h_n)f\}|\|_s&\leq&(1-\mu_s)\||\{(\beta_n
h_n)f\}|\|_s\\
&\leq&\big((1+\lambda_s)\|W\|_s(\sup\alpha_n)+\gamma_s\big)\|f\|_s,
\end{eqnarray*} By using $(2)$, we have:
\begin{eqnarray*}
(1+\mu_s)\||\{(\beta_n
h_n)f\}|\|_s&\geq&(1-\lambda_s)\||\{(\alpha_n
 g_n)f\}|\|_s-\gamma_s\|f\|_s\\
 &\geq&\big((1-\lambda_s)\|S\|_s^{-1}(\inf
 \alpha_n)-\gamma_s\big)\|f\|_s
  \end{eqnarray*}
for all $f\in X_F$ and $s\in\mathbb{N}$. Therefore
\begin{eqnarray*}
(1+\mu_s)\big(\sup\beta_n\big)\||\{(h_n)f\}|\|_s&\geq&(1+\mu_s)\||\{(\beta_n
h_n)f\}|\|_s\\
&\geq&\big((1-\lambda_s)\|S\|^{-1}\big(\inf\alpha_n\big)-\gamma_s\big)\|f\|_s
\end{eqnarray*} for all $f\in X_F$ and $s\in\mathbb{N}$. The
above inequality shows the frame bounds. Let $V=SD$. Then $V$ is a
bounded operator that $V\{h_n(f)\}=f$, for all $f\in X_F$.
Therefore $(\{h_i\}, V)$ is a Fr\'echet frame for $X_F$ with
respect to $\Theta_F$.
\end{proof}

\begin{thm}
Let $(\{g_i\}, V)$ be a Fr\'echet frame for $X_F$ with respect to
$\Theta_F$. Suppose $\lambda_{1_s},\lambda_{2_s},\mu_s\geq0$ such
that $\max\{\lambda_{2_s},\lambda_{1_s}+\mu_s B_s\}<1$ for all
$s\in\mathbb{N}$ and $S:\Theta_F\rightarrow X_F$ a continuous
operator such that for any $\{c_i\}\in\Theta_F$ and
$s\in\mathbb{N}$
\begin{equation}\label{7}
\|S\{c_i\}-V\{c_i\}\|_s\leq\lambda_1\|V\{c_i\}\|_s+\lambda_{2_s}\|S\{c_i\}\|_s+\mu_s\||\{c_i\}|\|_s
\end{equation}
then there exists a $\{h_i\}\subseteq X^*_F$ such that
$(\{h_i\},S)$ is a  Fr\'echet frame of $X_F$ with respect to
$\Theta_F$.
\end{thm}
\begin{proof}
For $f\in X_F$, let $c_i=g_i(f)$ in (\ref{7}), then we have
\begin{equation*}
\|S\{g_i(f)\}-V\{g_i(f)\}\|_s\leq\lambda_{1_s}\|V\{g_i(f)\}\|_s+\lambda_{2_s}\|S\{g_i(f)\}\|_s+\mu_s\||\{g_i(f)\}|\|_s
\end{equation*}
since $V(\{g_i(f)\})=f$, so
\begin{equation*}
\|S\{g_i(f)\}-f\|_s\leq\lambda_{1_s}\|f\}\|_s+\lambda_{2_s}\|S\{g_i(f)\}\|_s+\mu_s
B_S\|f|_s.
\end{equation*}
Let $L(f):=S\{g_i(f)\}$, so
$$
\|f-Lf\|_s\leq\lambda_{1_s}\|f\|_s+\lambda_{2_s}\|Lf\|_s+\mu_s
B_S\|f|_s
$$
or
$$
\|f-Lf\|_s\leq(\lambda_{1_s}+\mu_s
B_s)\|f\|_s+\lambda_{2_s}\|Lf\|_s.$$ Lemma \ref{1.1} results that
the operator $L$ is invertible and
$$
\frac{1-\lambda_{2_s}}{1+\lambda_{1_s}+\mu
B_s}\|f\|_s\leq\|L^{-1}f\|_s\leq\frac{1+\lambda_{2_s}}{1-(\lambda_{1_s}+\mu_s)
B_s}\|f\|_s
$$
and $f=LL^{-1}f=S(\{g_i(L^{-1}f)\})$. It is clear that
$\{g_i(L^{-1}f)\}\subseteq X^*_F$. By choosing $h_i=g_i \circ
L^{-1}$, we have
$$
\||\{h_i(f)\}|\|_s=\|| \{g_i(L^{-1}f)\}|\|_s\geq
A_s\||L^{-1}|\|_s\geq\frac{A_s(1-\lambda_{2_s})}{1+\lambda_{1_s}+\mu_s
B_s}\|f\|_s
$$
and
$$
\||\{h_i(f)\}|\|_s\leq B_s\|L^{-1}f\|_s\leq
\frac{B_s(1+\lambda_{2_s})}{1-(\lambda_{1_s}+\mu_s B_s)}\|f\|_s.
$$
\end{proof}
\begin{thm}
Let $(\{g_i\},S)$ be a F-frame for $X_F$ with respect to
$\Theta_F$. Let $\{h_i\}\subset X^*_F$. If for every
$s\in\mathbb{N}$ there exist  $\lambda_s,\mu_s\geq0$ such that
\begin{enumerate}
\item $\lambda_s\|U\|+\mu_s\leq\|S\|_s^{-1}$,
 \item
$\||\{(g_i-h_i)f\}|\|_{s}\leq\lambda_s\||\{(g_i)f\}|\|_s+\mu_s\|f\|_s$
for all $s\in\mathbb{N}$ and $f\in X_F$.
\end{enumerate}
Then there is a continuous operator $T$ such that $(\{h_i\}, T)$
is a F-frame for $X_F$ with respect to $\Theta_F$.
\end{thm}
\begin{proof} It is a direct result of $(2)$ that the operator
$V:X_F\rightarrow\Theta_F$ defined by $Vf:=\{h_i(f)\}$ is bounded
and
$$\||Uf-Vf|\|_s\leq\lambda_s\||Uf|\|_s+\mu_s\|f\|_s$$ for all
$s\in\mathbb{N}$ and $f\in X_F$. Therefore,
\begin{equation}\label{111}
\||Vf|\|_s\leq(\|U\|_s+\lambda_s\|U\|_s+\mu_s)\|f\|_s
\end{equation} for all $s\in\mathbb{N}$ and $f\in X_F$. Also, $SU=I$ imply
that
$$\|I-SV\|_s\leq\|S\|_s(\lambda_s\|U\|_s+\mu_s)<1.$$
Therefore $SV$ is invertible and
$\|(SV)^{-1}\|_s\leq\big(1-(\lambda\|U\|_s+\mu)\|S\|_s)\big)^{-1}$.
Finally $T=(SV)^{-1}S$ and $TV=I$. For all $f\in X_F$ and
$s\in\mathbb{N}$\begin{equation}\label{222}
\|f\|_s\leq\|T\|_s.\||Vf|\|_s\leq\frac{\|S\|}{1-(\lambda_s
\|U\|_s+\mu_s)\|S\|_s}\||Vf|\|_s.
\end{equation}
(\ref{111}) and (\ref{222}) which imply that for every
$s\in\mathbb{N}$ and $f\in X_F$:
$$
\frac{1-(\lambda_s \|U\|_s+\mu_s)\|S\|_s}{\|S\|_s}\|f\|_s\leq
\||\{h_i(f)\}|\|_{s}\leq(\|U\|_s+\lambda_s\|U\|_s+\mu_s)\|f\|_s.
$$
So $(\{h_i\}, T)$ is a F-frame for $X_F$ with respect to
$\Theta_F$.

\end{proof}

\textbf{Acknowledgment}: Some of the results in this paper were
obtained during the author's visiting at the  Acoustics Research
Institute, Austrian  Academy of Sciences, Austria ( Summers 2009,
2010 and 2011). He thanks this institute for hospitality. Also, he
would like to thank  Peter Balazs and D. T. Stoeva for useful
discussions, comments and suggestions. This work was partly
supported by the WWTF project MULAC ('Frame Multipliers: Theory
and Application in Acoustics; MA07-025)

\end{document}